\documentclass{amsproc}

\usepackage{ragged2e}

\newtheorem{theorem}{Theorem}[section]
\newtheorem{lemma}[theorem]{Lemma}

\newtheorem{proposition}[theorem]{Proposition}

\theoremstyle{definition}

\newtheorem*{definition*}{Definition}
\theoremstyle{remark}

\newtheorem*{remark*}{Remark}

\numberwithin{equation}{section}

\allowdisplaybreaks

\begin{document}

\title[Two estimates on the distribution of zeros of $L'(s,\chi)$ under GRH]{Two estimates on the distribution of zeros of the first derivative of Dirichlet $L$-functions under the generalized Riemann hypothesis}


\author{Ade Irma Suriajaya}
\address{Graduate School of Mathematics, Nagoya University, Furo-cho, Chikusa-ku, Nagoya, 464-8602, Japan}
\email{m12026a@math.nagoya-u.ac.jp}
\thanks{This work was partly supported by the Iwatani Naoji Foundation and JSPS KAKENHI Grant Number 15J02325.}


\begin{abstract}
The number of zeros and the distribution of the real part of non-real zeros of the derivatives of the Riemann zeta function have been investigated by B. C. Berndt, N. Levinson, H. L. Montgomery, H. Akatsuka, and the author.
Berndt, Levinson, and Montgomery investigated the unconditional case, while Akatsuka and the author gave sharper estimates under the truth of the Riemann hypothesis.
In this paper, we prove similar results related to the first derivative of Dirichlet $L$-functions associated with primitive Dirichlet characters under the assumption of the generalized Riemann hypothesis.
\end{abstract}

\maketitle

\section{Introduction}

Zeros of the Riemann zeta function are also related to those of its derivatives $\zeta^{(k)}(s)$ for positive integer $k$.
For example, A. Speiser \cite{spe} proved that the Riemann hypothesis is equivalent to the statement that $\zeta'(s)$ has no non-real zeros to the left of the critical line.
In 2012, assuming the Riemann hypothesis, H. Akatsuka \cite[Theorems 1 and 3]{aka} showed that we can approximate the distribution of zeros of $\zeta'(s)$ as follows:
\begin{align*}
\sum_{\substack{\rho'=\beta'+i\gamma',\\ \zeta'\left(\rho'\right)=0,\, 0<\gamma'\leq T}} \left(\beta'-\frac{1}{2}\right) &= \frac{T}{2\pi}\log{\log{\frac{T}{2\pi}}} + \frac{1}{2\pi}\left(\frac{1}{2}\log{2} - \log{\log{2}}\right)T\\
&\quad\quad- \operatorname{Li}\left(\frac{T}{2\pi}\right) + O((\log{\log{T})^2}),
\end{align*}
where the sum is counted with multiplicity and
$$ \operatorname{Li}(x) := \int_2^x \frac{dt}{\log{t}}, $$
and
\begin{equation} \label{eq:n1t}
N_1(T) = \frac{T}{2\pi}\log{\frac{T}{4\pi}}-\frac{T}{2\pi}+O\left(\frac{\log{T}}{(\log{\log{T}})^{1/2}}\right),
\end{equation}
where $N_1(T)$ denotes the number of zeros of $\zeta'(s)$ with $0<\operatorname{Im}(s)\leq T$, counted with multiplicity.
These results are also extended to higher order derivatives by the author \cite[Theorems 1 and 3]{ade}.

These results of Akatsuka \cite[Theorems 1 and 3]{aka} and the author \cite[Theorems 1 and 3]{ade}, under the truth of the Riemann hypothesis, improve the error terms $O(\log{T})$ in the unconditional results obtained by N. Levinson and H. L. Montgomery \cite[Theorem 10]{lev} on the distribution of the real part of zeros and by B. C. Berndt \cite[Theorem]{ber} on the number of zeros of $\zeta^{(k)}(s)$. \\
%

We are interested in extending these results of Akatsuka \cite[Theorems 1 and 3]{aka} to Dirichlet $L$-functions.
We shall only consider Dirichlet $L$-functions associated with primitive Dirichlet characters $\chi$ modulo $q>1$, $L(s,\chi)$.
Note that there exists only one Dirichlet character modulo $1$ and the associated Dirichlet $L$-function is the Riemann zeta function, whose results are given in \cite{aka}.
The {\it generalized Riemann hypothesis} states that both $\zeta(s)$ and $L(s,\chi)$ satisfy the Riemann hypothesis, that is, all nontrivial zeros lie on the critical line $\operatorname{Re}(s)=1/2$.

Zeros of $L^{(k)}(s,\chi)$ have been studied by C. Y. Y{\i}ld{\i}r{\i}m \cite{yil} in 1996 including zero-free regions and the number of zeros.
Akatsuka and the author in their recent preprint \cite[Theorems 1, 2, 4, and 5]{as} improved the zero-free region on the left half-plane \cite[Theorem 3]{yil} and the number of zeros \cite[Theorem 4]{yil} shown by Y{\i}ld{\i}r{\i}m, for the case $k=1$.
We also obtained a result \cite[Theorem 6]{as} on the distribution of the real part of zeros and proved results \cite[Theorems 8 and 9]{as}, analogous to Speiser's theorem \cite{spe}, for Dirichlet $L$-functions. \\

Throughout this paper, for a given integer $q>1$, we denote by $m$ the smallest prime number that does not divide $q$.
%
Next, we let $\rho = \beta + i\gamma$ and $\rho' = \beta' + i\gamma'$ denote the zeros of $L(s,\chi)$ and $L'(s,\chi)$ in the right half-plane $\operatorname{Re}(s)>0$.
We know that $L(s,\chi)$ has only trivial zeros in $\operatorname{Re}(s)\leq0$.
We remark that zeros of $L'(s,\chi)$ satisfying $\operatorname{Re}(s)\leq0$ can also be regarded as``trivial" (see \cite[Theorems 1, 2, and 4]{as}).
Then we define $N_1(T,\chi)$ for $T\geq2$ as the number of zeros of $L'(s,\chi)$ satisfying $\operatorname{Re}(s)>0$ and $|\operatorname{Im}(s)|\leq T$, counted with multiplicity.

Our main theorems are as follows:
\begin{theorem} \label{cha1'}
Assume that the generalized Riemann hypothesis is true, then for $T\geq2$, we have
\begin{align*}
\sum_{\substack{\rho' = \beta' + i\gamma',\\ |\gamma'| \leq T}} \left( \beta' - \frac{1}{2} \right)
&= \frac{T}{\pi} \log{\log{\frac{qT}{2\pi}}} + \frac{T}{\pi} \left( \frac{1}{2} \log{m} - \log{\log{m}} \right) - \frac{2}{q} \operatorname{Li}\left( \frac{qT}{2\pi} \right) \\
&\quad\quad+ O\left( m^{1/2}(\log{\log{(qT)}})^2 + m\log{\log{(qT)}} + m^{1/2}\log{q} \right),
\end{align*}
where the sum is counted with multiplicity.
\end{theorem}

\begin{theorem} \label{cha2'}
Assume that the generalized Riemann hypothesis is true, then for $T\geq2$, we have
$$
N_1(T,\chi)
= \frac{T}{\pi} \log{\frac{qT}{2m\pi}} - \frac{T}{\pi} + O\left( \frac{m^{1/2}\log{(qT)}}{(\log{\log{(qT)}})^{1/2}} + m^{1/2}\log{q} \right).
$$
\end{theorem}

In this paper, we first review some basic estimates related to $\log{L(s,\chi)}$ near the critical line and zero-free regions of $L'(s,\chi)$ in Section \ref{sec:l_prelim}.
In Section \ref{sec:l_keylems}, we show important lemmas crucial for the proofs of our main theorems
and finally prove them in Section \ref{sec:l_prfs}.
For convenience, we use variables $s$ and $z$ as complex numbers, with $\sigma = \operatorname{Re}(s)$ and $t = \operatorname{Im}(s)$.
Finally, we abbreviate the generalized Riemann hypothesis as GRH.


\section{Preliminaries}
\label{sec:l_prelim}


\subsection{Bounds related to $\log{L(s,\chi)}$ near the critical line}
\quad\\ \vskip-3mm

In this section we give some bounds related to $\log{L(s,\chi)}$ which can be found in \cite[Sections 12.1, 13.2, 14.1]{mon}.
Only for this subsection, we put $\tau := |t|+4$.

\begin{lemma} \label{mon13.2.6}
Assume GRH, then
$$
\log{L(\sigma+it,\chi)} = O\left( \frac{(\log{(q\tau)})^{2(1-\sigma)}}{(1-\sigma)\log{\log{(q\tau)}}} + \log{\log{\log{(q\tau)}}} \right)
$$
holds uniformly for $1/2 + (\log{\log{(q\tau)}})^{-1} \leq \sigma \leq 3/2$.
\end{lemma}

\begin{proof}
This is straightforward from the inequalities in exercise 6 of \cite[Section 13.2]{mon} (see also page 3 of \cite{mon.er} for the corrected exercise 6 (b) and (c)).
\end{proof}

%

%

\begin{lemma} \label{mon13.2.11}
Assume GRH, then
$$ \arg{L(\sigma+it,\chi)} = O\left( \frac{\log{(q\tau)}}{\log{\log{(q\tau)}}} \right) $$
holds uniformly for $\sigma \geq 1/2$.
\end{lemma}

\begin{proof}
See \cite[Section 5]{sel} or exercise 11 of \cite[Section 13.2]{mon}.
\end{proof}

With the above lemma and \cite[Corollary 14.6]{mon}, we obtain the following estimate on the number of zeros of $L(s,\chi)$ under GRH:
\begin{proposition} \label{ntchi_grh}
Assume GRH and let $N(T,\chi)$ denote the number of zeros of $L(s,\chi)$ satisfying $\operatorname{Re}(s)>0$ and $|\operatorname{Im}(s)|\leq T$, counted with multiplicity.
Then for $T\geq2$,
$$
N(T,\chi) = \frac{T}{\pi} \log{\frac{qT}{2\pi}} - \frac{T}{\pi} + O\left( \frac{\log{(qT)}}{\log{\log{(qT)}}} \right).
$$
\end{proposition}

\begin{proof}
This is a straightforward consequence of \cite[Corollary 14.6]{mon} and \cite[Theorem 6]{sel}
(see exercise 1 of \cite[Section 14.1]{mon}).
\end{proof}

\begin{lemma} \label{mon12.6}
$$
\frac{L'}{L}(\sigma+it,\chi) = \sum_{\substack{\rho = \beta +i\gamma,\\ |\gamma - t| \leq 1}} \frac{1}{\sigma+it-\rho} + O(\log{(q\tau)})
$$
holds uniformly for $-1 \leq \sigma \leq 2$.
\end{lemma}

\begin{proof}
See \cite[Lemma 12.6]{mon}.
\end{proof}


\subsection{Zero-free regions of $L'(s,\chi)$}
\quad\\ \vskip-3mm

We begin with a zero-free region of $L'(s,\chi)$ to the right of the critical line.
\begin{proposition} \label{zfr-right}
$L'(s,\chi)$ has no zeros when
$$ \sigma > 1 + \frac{m}{2} \left( 1 + \sqrt{1 + \frac{4}{m\log{m}} }\right). $$
\end{proposition}

\begin{proof}
See \cite[Theorem 2]{yil} for $k=1$.
\end{proof}

From the above proposition, it is not difficult to check that $L'(s,\chi) \neq 0$ when $\sigma \geq 1 + 3m/2$.
Next we introduce a zero-free region of $L'(s,\chi)$ to the left of the critical line.
\begin{proposition} \label{zfr-left}
$L'(s,\chi)$ has no zeros when $\sigma\leq0$ and $|t|\geq6$.
Furthermore, assuming GRH,
\begin{enumerate}
\item if $\kappa=0$ and $q\geq216$, then $L'(s,\chi)$ has a unique zero in $0<\operatorname{Re}(s)<1/2$;
\item if $\kappa=1$ and $q\geq23$, then $L'(s,\chi)$ has no zeros in $0<\operatorname{Re}(s)<1/2$.
\end{enumerate}
Here
\begin{align*}
\kappa =
\begin{cases}
0, & \chi(-1)=1; \\
1, & \chi(-1)=-1.
\end{cases}
\end{align*}
Thus under GRH, for any fixed $\epsilon>0$, there are only possibly finitely many zeros in the region defined by $0<\sigma<1/2$ and $|t|\leq\epsilon$ for any $L'(s,\chi)$.
\end{proposition}

\begin{proof}
See \cite[Theorems 1, 8, and 9]{as} and note that $q\geq3$ in our case.
\end{proof}


\section{Key lemmas}
\label{sec:l_keylems}

For convenience, we define the function $F(s,\chi)$ as follows:
\begin{equation} \label{eq:F}
F(s,\chi) := \epsilon(\chi)2^s\pi^{s-1}q^{\frac{1}{2}-s}\sin{\left(\frac{\pi(s+\kappa)}{2}\right)}\Gamma(1-s),
\end{equation}
where $\epsilon(\chi)$ is a factor that depends only on $\chi$, satisfying $|\epsilon(\chi)|=1$, and 
$\kappa$ is determined as in Proposition \ref{zfr-left}.
Thus from the functional equation for $L(s,\chi)$, we have $L(s,\chi) = F(s,\chi)L(1-s,\overline{\chi})$.
We also define the function $G_1(s,\chi)$ associated with $L'(s,\chi)$ as follows:
\begin{equation} \label{eq:G_1}
G_1(s,\chi) := -\frac{m^s}{\chi(m)\log{m}} L'(s,\chi).
\end{equation}

\subsection{Constants $\sigma_1$ and $t_q$}
\label{sec:consts}

\begin{lemma} \label{lem1.1}
For $\sigma\geq2$, we have
$$
|G_1(\sigma+it,\chi) - 1| \leq 2 \left( 1+ \frac{8m}{\sigma} \right) \left( 1+ \frac{1}{m} \right)^{-\sigma}
$$
and
$$
\left| \frac{G_1}{L}(\sigma+it,\chi) - 1 \right| \leq 2 \left( 1+ \frac{8m}{\sigma} \right) \left( 1+ \frac{1}{m} \right)^{-\sigma}
$$
\end{lemma}
\begin{proof}
Let $\sigma\geq2$. Then from \eqref{eq:G_1} and by using the Dirichlet series expression of $L'(s,\chi)$, we can calculate
\begin{align*}
|G_1(s,\chi) - 1|
&= \left| -\frac{m^s}{\chi(m)\log{m}} \left( -\sum_{n=1}^\infty \frac{\chi(n)\log{n}}{n^s} \right) - 1 \right| \\
&= \left| \frac{m^s}{\chi(m)\log{m}} \sum_{n=m+1}^\infty \frac{\chi(n)\log{n}}{n^s} \right|
\leq \frac{m^\sigma}{\log{m}} \sum_{n=m+1}^\infty \frac{\log{n}}{n^\sigma} \\
&\leq \frac{m^\sigma}{\log{m}} \frac{\log{(m+1)}}{(m+1)^\sigma} + \frac{m^\sigma}{\log{m}} \int_{m+1}^\infty \frac{\log{x}}{x^\sigma} dx \\
&= \frac{m^\sigma}{\log{m}} \frac{\log{(m+1)}}{(m+1)^\sigma} \left( 1 + \frac{m+1}{\sigma-1} + \frac{m+1}{(\sigma-1)^2\log{(m+1)}} \right) \\
&\leq \frac{m^\sigma}{\log{m}} \frac{2\log{m}}{(m+1)^\sigma} \left( 1 + \frac{4m}{\sigma-1} \right)
\leq 2 \left( \frac{m}{m+1} \right)^\sigma \left( 1 + \frac{8m}{\sigma} \right),
\end{align*}
where we have used $m+1 \leq 2m \leq m^2$ and $\sigma-1\geq\sigma/2$ in the last two inequalities.

By using the Dirichlet series expansion of $(L'/L)(s,\chi)$, with similar calculation as the above, we can show the second inequality in the lemma.
\end{proof}


Applying Stirling's formula of the following form
\begin{equation} \label{eq:stirling}
\log{\Gamma(z)} = \left(z - \frac{1}{2}\right)\log{z} - z + \frac{1}{2}\log{2\pi} + \int_0^\infty \frac{[u] - u + \frac{1}{2}}{u + z} du
\end{equation}
\flushright
($-\pi + \delta \leq \arg{z} \leq \pi - \delta$, for any $\delta > 0$), \\
\justify
we can define the holomorphic function
\begin{equation} \label{eq:logF}
\begin{aligned}
\log{F(s,\chi)}
&:= \log{\epsilon(\chi)} + \left( \frac{1}{2} - s \right) \log{\frac{q}{2\pi}} + \frac{1}{2} \log{\frac{2}{\pi}}
+ \log{\sin{\frac{\pi}{2}(s+\kappa)}} \\
&\quad\quad+ \log{\Gamma(1-s)}
\end{aligned}
\end{equation}
for $\sigma<1$ and $|t|>1$, where $0 \leq \arg{\epsilon(\chi)} < 2\pi$ and $\log{\sin{\frac{\pi}{2}(s+\kappa)}}$ is the holomorphic function on $\sigma<1, |t|>1$ satisfying
\begin{align*}
\log{\sin{\frac{\pi}{2}(s+\kappa)}} :=
\begin{cases}
\displaystyle \frac{(1-s-\kappa)\pi}{2}i - \log{2} - \sum_{n=1}^\infty \frac{e^{\pi i(s+\kappa)n}}{n},
& t>1; \\\\
\displaystyle \frac{(s+\kappa-1)\pi}{2}i - \log{2} - \sum_{n=1}^\infty \frac{e^{-\pi i(s+\kappa)n}}{n},
& t<-1.
\end{cases}
\end{align*}

Under the above definitions, we can show the following lemma.
\begin{lemma} \label{logdevF}
For $\sigma < 1$ and $\pm t > 1$, we have
$$
\frac{F'}{F}(s,\chi) = -\log{(q(1-s))} + \log{2\pi} \mp \frac{\pi i}{2} + \frac{1}{2(1-s)} + O\left( \frac{1}{|1-s|^2} \right) + O\left( e^{-\pi|t|} \right),
$$
where $-\pi/2 < \arg{(1-s)} < \pi/2$.
\end{lemma}

\begin{proof}
Applying Stirling's formula \eqref{eq:stirling} to $\log{\Gamma(z)}$ for $\arg{z} \in (-\pi/2,\pi/2)$, we have
$$
\log{\Gamma(1-s)} = \left( \frac{1}{2} - s \right) \log{(1-s)} - (1-s) + \frac{1}{2}\log{2\pi} + \int_0^\infty \frac{[u] - u + \frac{1}{2}}{u + 1 - s} du
$$
in the region $\sigma<1,\, |t|>1$.
From \eqref{eq:logF}, we can show that
\begin{align*}
\log{F(s,\chi)}
&= \log{\epsilon(\chi)} + \frac{\pi}{2} \left( \frac{1}{2} - \kappa \right)i - 1
+ \left( \frac{1}{2} - s \right) \left( \log{(q(1-s))} - \log{2\pi} + \frac{\pi i}{2} \right) \\
&\quad\quad+ s + \int_0^\infty \frac{[u] - u + \frac{1}{2}}{u + 1 - s} du
- \sum_{n=1}^\infty \frac{e^{\pi i(s+\kappa)n}}{n}
\end{align*}
holds when $\sigma<1$ and $t>1$.
Differentiating both sides of the above equation with respect to $s$, we obtain
$$
\frac{F'}{F}(s,\chi) = -\log{(q(1-s))} + \log{2\pi} - \frac{\pi i}{2} + \frac{1}{2(1-s)} + O\left( \frac{1}{|1-s|^2} \right) + O\left( e^{-\pi|t|} \right)
$$
for $\sigma<1$ and $t>1$.
We can show similarly for $\sigma<1$ and $t<-1$.
\end{proof}


\begin{lemma} \label{lem1.2}
There exists a $\sigma_1\leq-1$ such that
$$ \left| \frac{1}{\frac{F'}{F}(s,\chi)} \frac{L'}{L}(1-s,\overline{\chi}) \right| < 2^\sigma $$
holds for any $s$ with $\sigma\leq\sigma_1$ and $|t|\geq2$.
\end{lemma}

\begin{proof}
From Lemma \ref{logdevF}, we know that
$$ \frac{F'}{F}(s,\chi) = -\log{(q(1-s))} + O(1) $$
holds when $\sigma<1$ and $|t|\geq2$.
Hence
$$ \left| \frac{F'}{F}(s,\chi) \right| \geq \log{(q(1-\sigma))} - |O(1)| $$
holds in the region $\sigma<1, |t|\geq2$.
Thus, we can take $\sigma'_1\leq-1$ sufficiently small (i.e. sufficiently large in the negative direction) so that for any $s$ with $\sigma\leq\sigma'_1$ and $|t|\geq2$, we have
\begin{equation} \label{eq:logdevF}
\left| \frac{F'}{F}(s,\chi) \right| \geq \frac{1}{2}\log{(q(1-\sigma))} 
\end{equation}
for all $s$ in the region $\sigma\leq\sigma'_1, |t|\geq2$.

Next we estimate $(L'/L)(1-s,\overline{\chi})$.
In the region $\sigma\leq-1,\, |t|\geq2$, $(L'/L)(1-s,\overline{\chi})$ can be written as a Dirichlet series, thus we have
\begin{equation} \label{eq:logdevL}
\begin{aligned}
\left| \frac{L'}{L}(1-s,\overline{\chi}) \right|
&\leq \frac{\log{2}}{2^{1-\sigma}} + \sum_{n=3}^\infty \frac{\log{n}}{n^{1-\sigma}}
\leq \frac{2^\sigma \log{2}}{2} + \int_2^\infty \frac{\log{x}}{x^{1-\sigma}} dx \\
&= 2^\sigma \left( \frac{\log{2}}{2} - \frac{\log{2}}{\sigma} + \frac{1}{\sigma^2} \right)
\leq 2^\sigma \left( 1+ \frac{3}{2}\log{2} \right).
\end{aligned}
\end{equation}

Now combining inequalities \eqref{eq:logdevF} and \eqref{eq:logdevL}, we have
$$
\left| \frac{1}{\frac{F'}{F}(s,\chi)}\frac{L'}{L}(1-s,\overline{\chi}) \right| < 2^\sigma \frac{2+3\log{2}}{\log{(q(1-\sigma))}}
$$
for $\sigma\leq\sigma'_1$ and $|t|\geq2$.
Hence we can find some $\sigma_1\leq\sigma'_1$ ($\leq-1$) such that $(2+3\log{2})/\log{(q(1-\sigma))} < 1$
holds for any $\sigma\leq\sigma_1$.
This implies that
$$ \left| \frac{1}{\frac{F'}{F}(s,\chi)}\frac{L'}{L}(1-s,\overline{\chi}) \right| < 2^\sigma $$
holds in the region $\sigma\leq\sigma_1, |t|\geq2$.
\end{proof}

\begin{lemma} \label{lem1.3}
Assume GRH and fix a $\sigma_1$ that satisfies Lemma \ref{lem1.2}. Then there exists a $t_1>-\sigma_1$ such that
\begin{enumerate}
\item
for any $s$ satisfying $\sigma_1 \leq \sigma \leq 1/2$ and $|t| \geq t_1 - 1$,
$$ \left| \frac{F'}{F}(s,\chi) \right|\geq 1 $$
holds and we can take the logarithmic branch of $\log{(F'/F)(s,\chi)}$ in that region such that it is holomorphic there and $5\pi/6 < \arg{(F'/F)(s,\chi)} < 7\pi/6$ holds;
\item
for any $s$ satisfying $\sigma_1 \leq \sigma < 1/2$ and $|t| \geq t_1 - 1$,
$$ \frac{L'}{L}(s,\chi)\neq0 $$
holds and we can take the logarithmic branch of $\log{(L'/L)(s,\chi)}$ in that region such that it is holomorphic there and $\pi/2 < \arg{(L'/L)(s,\chi)} < 3\pi/2$ holds.
\end{enumerate}
\end{lemma}

\begin{proof}
We begin by examining condition (1).
Again, from Lemma \ref{logdevF}, we see that
$$ \frac{F'}{F}(s,\chi) = -\log{(q(1-s))} + O(1) $$
holds when $\sigma<1$ and $|t|\geq2$.
Thus for $\sigma_1 \leq \sigma \leq 1/2$ and $|t| \geq 2$, we have
$$ \left| \frac{F'}{F}(s,\chi) \right| \geq \log{(q|t|)} - |O(1)| \geq \log{|t|} - |O(1)|. $$
Hence, we can find some $t'_1\geq100$ such that
\begin{equation} \label{eq:logdevF_1}
\left| \frac{F'}{F}(s,\chi) \right| \geq 1
\end{equation}
holds for all $s$ with $\sigma_1 \leq \sigma \leq 1/2$ and $|t| \geq t'_1 - 1$.
We note that Lemma \ref{logdevF} also implies that 
$$ \frac{F'}{F}(s,\chi) = -\log{(q|t|)} + O(1) $$
holds when $\sigma_1 \leq \sigma \leq 1/2$ and $|t| \geq 2 - \sigma_1$.
Consequently, we can find some $t''_1 \geq \max{\{ t'_1, 3 - \sigma_1 \}}$ such that
$$
\frac{5\pi}{6} < \arg{\frac{F'}{F}(s,\chi)} < \frac{7\pi}{6}
$$
holds for $\sigma_1 \leq \sigma \leq 1/2$ and $|t| \geq t''_1 - 1$.
Since $(F'/F)(s,\chi)$ is holomorphic, inequality \eqref{eq:logdevF_1} tells us that $\log{(F'/F)(s,\chi)}$ is holomorphic in the region $\sigma_1 \leq \sigma \leq 1/2, |t| \geq t''_1 - 1$ with this branch.

By the above calculations, we find that $t''_1$ is a candidate for $t_1$.
Below we examine condition (2) to completely prove the existence of $t_1$.

Corollary 10.18 of \cite{mon} allows us to show that
$$
\operatorname{Re} \left( \frac{L'}{L}(s,\chi) \right) < -\frac{1}{2} \log{\frac{q}{\pi}} - \frac{1}{2} \,\operatorname{Re} \left( \frac{\Gamma'}{\Gamma}\left( \frac{s+\kappa}{2} \right) \right)
$$
holds for $\sigma_1 \leq \sigma < 1/2$, under GRH.
For any small $\delta>0$, let $|t| > \sigma_1 \tan{\delta}$. Stirling's formula \eqref{eq:stirling} implies
$$
\frac{1}{2} \,\operatorname{Re} \left( \frac{\Gamma'}{\Gamma}\left( \frac{s+\kappa}{2} \right) \right) = \frac{1}{2} \log{\left| \frac{s+\kappa}{2} \right|} + O\left( \frac{1}{|s|} \right).
$$
Hence we can find some $t_1 \geq t''_1$ large enough so that
$$
\operatorname{Re} \left( \frac{L'}{L}(s,\chi) \right) < 0
$$
holds for $\sigma_1 \leq \sigma < 1/2$ and $|t| \geq t_1 - 1$ and hence $ (L'/L)(s,\chi) \neq 0 $.
Moreover, we can define a branch of $\log{(L'/L)(s,\chi)}$ so that it is holomorphic in $\sigma_1 \leq \sigma < 1/2$, $|t| \geq t_1 - 1$ and
$$
\frac{\pi}{2} < \arg{\frac{L'}{L}(s,\chi)} < \frac{3\pi}{2}
$$
holds there.
Since this $t_1$ also satisfies condition (1), the proof is complete.
\end{proof}

Now we fix $t_1$ which satisfies Lemma \ref{lem1.3} and take $t_q \in [t_1+1, t_1+2]$ such that
\begin{equation} \label{eq:t_q}
L(\sigma \pm it_q,\chi)\neq0,\, L'(\sigma \pm it_q,\chi)\neq0
\end{equation}
for all $\sigma\in\mathbb{R}$.

\begin{remark*}
We note that $t_q$ depends on $q$ but it is bounded by a fixed constant that does not depend on $q$: $t_q \ll t_1 \ll 1$.
\end{remark*}


\subsection{Bounds related to $\log{G_1(s,\chi)}$}
\label{sec:lems3-5}
\quad\\ \vskip-3mm

In this subsection, we give bounds for $\arg{(G_1/L)(s,\chi)}$ and $\arg{G_1(s,\chi)}$.
We take the logarithmic branches so that $\log{L(s,\chi)}$ and $\log{G_1(s,\chi)}$ tend to $0$ as $\sigma\rightarrow\infty$ and are holomorphic in
$\mathbb{C} \backslash \{ \rho+\lambda \mid L(\rho,\chi)=0, \lambda \leq 0 \}$ and
$\mathbb{C} \backslash \{ \rho'+\lambda \mid L'(\rho',\chi)=0, \lambda \leq 0 \}$, respectively.
We write
$$
-\arg{L(\sigma\pm i\tau,\chi)}+\arg{G_1(\sigma\pm i\tau,\chi)} = \arg{\frac{G_1}{L}(\sigma\pm i\tau,\chi)}
$$
and take the argument on the right-hand side so that
$\log{(G_1/L)(s,\chi)}$ tends to $0$ as $\sigma\rightarrow\infty$ and is holomorphic in
$\mathbb{C} \backslash \{ z+\lambda \mid (L'/L)(z,\chi)=0 \text{ or } \infty, \lambda\leq0 \}$.


\begin{lemma} \label{lem3'}
Assume GRH and let $\tau\geq t_q$.
Then we have for $1/2 < \sigma \leq 10m$,
$$
\arg{\frac{G_1}{L}(\sigma\pm i\tau,\chi)} \ll
\begin{cases}
\displaystyle \frac{m}{\sigma} & 3\leq\sigma\leq10m, \\\\
\displaystyle \frac{m^{1/2} \log{\log{(q\tau)}} + m}{\sigma-1/2} & 1/2<\sigma\leq3.
\end{cases}
$$
\end{lemma}

\begin{proof}
Let $\tau \geq t_q$ and $1/2 < \sigma \leq 10m$.
Let
$$
u_{G_1/L} = u_{G_1/L}(\sigma, \tau; \chi) := \#\left\{ u \in [\sigma, 11m] \mid \operatorname{Re}\left( \frac{G_1}{L}(u \pm i\tau, \chi) \right) = 0 \right\},
$$
then
$$ \left| \arg{\frac{G_1}{L}(\sigma \pm i\tau, \chi)} \right| \leq \left( u_{G_1/L} + 1 \right) \pi. $$
To estimate $u_{G_1/L}$, we set
$$
H_1(z,\chi) := \frac{1}{2} \left( \frac{G_1}{L}(z \pm i\tau, \chi) + \frac{G_1}{L}(z \mp i\tau, \overline{\chi}) \right)
$$
and
$$
n_{H_1}(r,\chi) := \#\{ z\in\mathbb{C} \mid H_1(z,\chi) = 0, |z - 11m| \leq r \}.
$$
Since $H_1(x,\chi) = \operatorname{Re}( (G_1/L)(x \pm i\tau, \chi) )$ for $x\in\mathbb{R}$, we have $u_{G_1/L} \leq n_{H_1}(11m-\sigma,\chi)$ for $1/2 < \sigma \leq 10m$.

Now we estimate $n_{H_1}(11m-\sigma,\chi)$.
We take $\epsilon = \epsilon_{\sigma, \tau} > 0$.
It is easy to show that
$$
n_{H_1}(11m-\sigma,\chi) \leq \frac{1}{\log{\left(1+\epsilon/(11m-\sigma)\right)}} \int_0^{11m-\sigma+\epsilon} \frac{n_{H_1}(r,\chi)}{r} dr.
$$
Applying Jensen's theorem (cf. \cite[Section 3.61]{tit1}), we have
\begin{align*}
\int_0^{11m-\sigma+\epsilon} \frac{n_{H_1}(r,\chi)}{r} dr &=
\frac{1}{2\pi} \int_0^{2\pi} \log{ |H_1( 11m + (11m-\sigma+\epsilon)e^{i\theta}, \chi )| } d\theta \\
&\quad\quad- \log{|H_1(11m, \chi)|}.
\end{align*}
Applying the second inequality in Lemma \ref{lem1.1}, we can easily see that $\log{|H_1(11m, \chi)|}=O(1)$.
Therefore
\begin{align*}
\left| \arg{\frac{G_1}{L}(\sigma \pm i\tau, \chi)} \right|
&\leq \frac{1}{\log{(1+\epsilon/(11m-\sigma))}} \\
&\quad\quad\times \left( \frac{1}{2\pi} \int_0^{2\pi} \log{ |H_1( 11m+(11m-\sigma+\epsilon)e^{i\theta}, \chi )| } d\theta + C \right)
\end{align*}
for some absolute constant $C>0$.

Now we divide the rest of the proof in two cases:
\begin{enumerate}
\item[(a)] For $3\leq\sigma\leq10m$, we restrict $\epsilon$ to satisfy $0<\epsilon\leq\sigma-2$. Then $11m+(11m-\sigma+\epsilon)\cos{\theta} \geq 2$.
Applying the second inequality in Lemma \ref{lem1.1}, we can easily obtain
$$
|H_1( 11m+(11m-\sigma+\epsilon)e^{i\theta}, \chi )| \leq \frac{100m}{11m+(11m-\sigma+\epsilon)\cos{\theta}}.
$$
By using Jensen's theorem (cf. \cite[Section 3.61]{tit1}), we can show that for $c>r>0$,
\begin{equation} \label{eq:log.cos}
\frac{1}{2\pi} \int_0^{2\pi} \log{|c + r \cos{\theta}|} d\theta = \log{\frac{c+\sqrt{c^2-r^2}}{2}}
\end{equation}
holds.
By using \eqref{eq:log.cos}, we can easily show that
\begin{align*}
\frac{1}{2\pi} &\int_0^{2\pi} \log{ |H_1( 11m+(11m-\sigma+\epsilon)e^{i\theta}, \chi )| } d\theta \\
&\quad\quad\leq \log{(100m)} - \frac{1}{2\pi} \int_0^{2\pi} \log{(11m+(11m-\sigma+\epsilon)\cos{\theta})} d\theta \\
&\quad\quad= \log{(100m)} - \log{\frac{11m+\sqrt{11m^2-(11m-\sigma+\epsilon)^2}}{2}} \\
&\quad\quad\leq \log{(100m)} - \log{\frac{11m}{2}} \ll 1.
\end{align*}
Note that $\epsilon/(11m-\sigma) \leq 10$, thus $\log{(1+\epsilon/(11m-\sigma))} \geq \epsilon/(11m-\sigma)$.
Hence
$$ \arg{\frac{G_1}{L}(\sigma \pm i\tau, \chi)} \leq \frac{11m-\sigma}{\epsilon} \ll \frac{m}{\epsilon}. $$
By taking $\epsilon=\sigma-2$, we obtain
$$ \arg{\frac{G_1}{L}(\sigma \pm i\tau, \chi)} \ll \frac{m}{\sigma}. $$
This is the first inequality in Lemma \ref{lem3'}.

\item[(b)] For $1/2<\sigma\leq3$, we restrict $\epsilon$ to satisfy $0<\epsilon<\sigma-1/2$ and we divide the interval of integration into
\begin{itemize}
\item $ \mathcal{I}_1 := \{ \theta\in[0,2\pi] \mid 11m+(11m-\sigma+\epsilon)\cos{\theta} \geq 2 \} $ and
\item $ \mathcal{I}_2 := \{ \theta\in[0,2\pi] \mid 11m+(11m-\sigma+\epsilon)\cos{\theta} < 2 \} $.
\end{itemize}
Since $11m+(11m-\sigma+\epsilon)\cos{\theta}>1/2$ and $11m-\sigma+\epsilon<11m$, on $\mathcal{I}_1$, as in the calculation of case (a), we can show that
\begin{align*}
\frac{1}{2\pi} &\int_{\theta\in\mathcal{I}_1} \log{ |H_1( 11m+(11m-\sigma+\epsilon)e^{i\theta}, \chi )| } d\theta \\
&\quad\quad\leq \frac{1}{2\pi} \int_{\theta\in\mathcal{I}_1} \log{ \frac{100m}{11m+(11m-\sigma+\epsilon)\cos{\theta}} } d\theta \\
&\quad\quad\leq \frac{1}{2\pi} \int_0^{2\pi} \log{ \frac{100m}{11m+(11m-\sigma+\epsilon)\cos{\theta}} } d\theta \ll 1.
\end{align*}

Now we estimate the integral on $\mathcal{I}_2$.
Setting
$$ \cos{\theta_0} := \frac{11m-2}{11m-\sigma+\epsilon} $$
for $\theta_0\in(0,\pi/2)$, we have $\mathcal{I}_2 = (\pi-\theta_0, \pi+\theta_0)$.
Applying Lemma \ref{mon12.6} and Proposition \ref{ntchi_grh}, and noting that $(L'/L)(x+iy,\chi) = O(1)$ when $x\geq2$, we have
$$ \frac{L'}{L}(x+iy, \chi) = O\left( \frac{\log{(q(|y|+1))}}{x-1/2} \right) $$
for $1/2<x\leq A$, for any fixed $A\geq2$.
Thus,
$$
|H_1( 11m + (11m-\sigma+\epsilon)e^{i\theta}, \chi )|
\leq C_1\frac{m^2}{\log{m}} \frac{\log{(q(\tau+11m))}}{11m + (11m-\sigma+\epsilon)\cos{\theta} - 1/2}
$$
for some absolute constant $C_1>0$.
Hence
\begin{align*}
&\frac{1}{2\pi} \int_{\theta\in\mathcal{I}_2} \log{ |H_1( 11m + (11m-\sigma+\epsilon)e^{i\theta}, \chi )| } d\theta \\
&\quad\quad\leq \frac{1}{2\pi} \int_{\pi-\theta_0}^{\pi+\theta_0} \log{ \frac{C_1 m^2}{\log{m}} \frac{\log{(q(\tau+11m))}}{11m-1/2 + (11m-\sigma+\epsilon)\cos{\theta}} } d\theta \\
&\quad\quad= \frac{1}{2\pi} \int_{-\theta_0}^{\theta_0} \log{ \frac{C_1 m^2}{\log{m}} \frac{\log{(q(\tau+11m))}}{11m-1/2 - (11m-\sigma+\epsilon)\cos{\theta}} } d\theta \\
&\quad\quad= 
\frac{\theta_0}{\pi} \log{ \frac{C_1 m^2 \log{(q(\tau+11m))}}{\log{m}} } \\
&\quad\quad\quad\quad- \frac{1}{2\pi} \int_{-\theta_0}^{\theta_0} \log{ \left(11m-\frac{1}{2} + (11m-\sigma+\epsilon)\cos{\theta}\right) } d\theta.
\end{align*}

We note that $\cos{\theta_0} = 1+ O(1/m)$.
By using $1-\cos{\theta_0}=2\sin^2{(\theta_0/2)}$, we can show
$$ \theta_0 \ll \left| \sin^2{\frac{\theta_0}{2}} \right| \ll \frac{1}{m^{1/2}}. $$
Hence,
\begin{align*}
&\int_{-\theta_0}^{\theta_0} \log{ \left(11m-\frac{1}{2} + (11m-\sigma+\epsilon)\cos{\theta}\right) } d\theta \\
&\quad\quad= \int_{-\theta_0}^{\theta_0} \log{ \frac{11m-1/2 + (11m-\sigma+\epsilon)\cos{\theta}}{11m-1/2} } d\theta + \int_{-\theta_0}^{\theta_0} \log{(11m-1/2)} d\theta \\
&\quad\quad= \int_{-\theta_0}^{\theta_0} \log{\left( 1- \frac{11m-\sigma+\epsilon}{11m-1/2}\cos{\theta} \right)} d\theta + O\left( \frac{\log{m}}{m^{1/2}} \right)
\end{align*}
Recalling that $\sigma-\epsilon>1/2$ and $\theta_0\in(0,\pi/2)$, we have
$$
\int_{-\theta_0}^{\theta_0} \log{\left( 1-\cos{\theta} \right)} d\theta
\leq \int_{-\theta_0}^{\theta_0} \log{\left( 1- \frac{11m-\sigma+\epsilon}{11m-1/2}\cos{\theta} \right)} d\theta
\leq 0.
$$
Meanwhile,
\begin{align*}
\int_{-\theta_0}^{\theta_0} \log{\left( 1-\cos{\theta} \right)} d\theta
&= \int_{-\theta_0}^{\theta_0} \log{\left( 2\sin^2{\frac{\theta}{2}} \right)} d\theta
= 2\theta_0 \log{2} + 4 \int_0^{\theta_0} \log{\left( \sin{\frac{\theta}{2}} \right)} d\theta \\
&= 2\theta_0 \log{2} + 4 \int_0^{\theta_0} \log{\frac{\sin{(\theta/2)}}{\theta/2}} d\theta + 4 \int_0^{\theta_0} \log{\frac{\theta}{2}} d\theta \\
&= O\left( \theta_0 \right) + O\left( \theta_0^3 \right) + O\left( \theta_0\log{\theta_0^{-1}} \right)
= O\left( \frac{\log{m}}{m^{1/2}} \right).
\end{align*}

Therefore when $1/2<\sigma\leq3$, we have
\begin{align*}
\frac{1}{2\pi} \int_0^{2\pi} &\log{ |H_1( 11m+(11m-\sigma+\epsilon)e^{i\theta}, \chi )| } d\theta \\
&\quad\quad= \frac{1}{2\pi} \left( \int_{\theta\in\mathcal{I}_1} + \int_{\theta\in\mathcal{I}_2} \right) \log{ |H_1( 11m+(11m-\sigma+\epsilon)e^{i\theta}, \chi )| } d\theta \\
&\quad\quad\ll 1 + \frac{\log{\log{(q(\tau+11m))}}}{m^{1/2}} + \frac{\log{m}}{m^{1/2}}
\ll 1 + \frac{\log{\log{(q\tau)}}}{m^{1/2}}.
\end{align*}
Since $0 < \epsilon/(11m-\sigma) < 1$, we have $\log{(1+\epsilon/(11m-\sigma))} \gg \epsilon/m$, thus
$$
\arg{\frac{G_1}{L}(\sigma \pm i\tau, \chi)} \ll \frac{m}{\epsilon} \left( 1 + \frac{\log{\log{(q\tau)}}}{m^{1/2}} \right).
$$
Taking $\epsilon=(\sigma-1/2)/2$,
we obtain the second inequality in Lemma \ref{lem3'}.
\end{enumerate}
\end{proof}


\begin{lemma} \label{lem4'}
Assume GRH and let $A\geq2$ be fixed. Then there exists a constant $C_0>0$ such that
$$
\left| L'(\sigma+it,\chi) \right| \leq \exp{ \left( C_0\left( \frac{(\log q\tau)^{2(1-\sigma)}}{\log\log{(q\tau)}} + (\log{(q\tau)})^{1/10} \right) \right) }
$$
holds for $1/2 - 1/\log{\log{(q\tau)}} \leq \sigma \leq A$ and $\tau = |t|+4$.
\end{lemma}

\begin{proof}
Applying Lemma \ref{mon13.2.6} and Cauchy's integral formula, Lemma \ref{lem4'} follows.
\end{proof}


\begin{lemma} \label{lem5'}
Assume GRH.
Then for any $1/2 \leq \sigma \leq 3/4$, we have
\begin{align*}
\arg{G_1(\sigma\pm i\tau,\chi)} &= O\bigg( m^{1/2}(\log{\log{(q\tau)}}) \\
&\quad\quad\times \left( m^{1/2} + (\log{(q\tau)})^{1/10} + \frac{(\log{(q\tau)})^{2(1-\sigma)}}{(\log{\log{(q\tau)}})^{3/2}}\right) \bigg).
\end{align*}
\end{lemma}

\begin{proof}
The proof is similar to that of Lemma \ref{lem3'} but we provide the details for clarity.
Let $1/2 \leq \sigma \leq 3/4$ and $\tau>1$ be large.
Put
$$
u_{G_1} = u_{G_1}(\sigma, \tau; \chi) := \#\left\{ u \in [\sigma, 1+3m/2] \mid \operatorname{Re}\left( G_1(u \pm i\tau, \chi) \right) = 0 \right\},
$$
then
$$ \left| \arg{G_1(\sigma \pm i\tau, \chi)} \right| \leq \left( u_{G_1} + 1 \right) \pi. $$
To estimate $u_{G_1}$, we set
$$
X_1(z,\chi) := \frac{G_1(z \pm i\tau, \chi) + G_1(z \mp i\tau, \overline{\chi})}{2}
$$
and
$$
n_{X_1}(r,\chi) := \#\{ z\in\mathbb{C} \mid X_1(z,\chi) = 0, |z - (1+3m/2)| \leq r \}.
$$
Then we have $u_{G_1} \leq n_{X_1}(1+3m/2-\sigma,\chi)$.

Now we estimate $n_{X_1}(1+3m/2-\sigma,\chi)$.
For each $\sigma \in [1/2, 3/4]$, we take $\epsilon = \epsilon_{\sigma, \tau}$ satisfying $0 < \epsilon \leq \sigma - 1/2 + (\log{\log{(q\tau)}})^{-1}$.
It is easy to show that
$$
n_{X_1}(1+3m/2-\sigma,\chi) \leq \frac{1+3m}{\epsilon} \int_0^{1+3m/2-\sigma+\epsilon} \frac{n_{X_1}(r,\chi)}{r} dr.
$$
Applying Jensen's theorem (cf. \cite[Section 3.61]{tit1}), we have
\begin{align*}
&\int_0^{1+3m/2-\sigma+\epsilon} \frac{n_{X_1}(r,\chi)}{r} dr \\
&\quad\quad= \frac{1}{2\pi} \int_0^{2\pi} \log{ |X_1( 1+3m/2 + (1+3m/2-\sigma+\epsilon)e^{i\theta}, \chi )| } d\theta \\
&\quad\quad\quad\quad - \log{|X_1(1+3m/2, \chi)|}.
\end{align*}
By using the first inequality in Lemma \ref{lem1.1}, we can easily show
$$ \log{|X_1(1+3m/2, \chi)|} = O(1). $$

As in the proof of Lemma \ref{lem3'}, we divide the interval of integration into
\begin{itemize}
\item $ \mathcal{J}_1 := \{ \theta\in[0,2\pi] \mid 1+3m/2+(1+3m/2-\sigma+\epsilon)\cos{\theta} \geq 2 \} $ and
\item $ \mathcal{J}_2 := \{ \theta\in[0,2\pi] \mid 1+3m/2+(1+3m/2-\sigma+\epsilon)\cos{\theta} < 2 \} $.
\end{itemize}
Then similarly, applying the first inequality in Lemma \ref{lem1.1} and \eqref{eq:log.cos}, we can show that
$$
\frac{1}{2\pi} \int_{\theta\in\mathcal{J}_1} \log{ |X_1( 1+3m/2+(1+3m/2-\sigma+\epsilon)e^{i\theta}, \chi )| } d\theta = O(1).
$$
Next we estimate the integral on $\mathcal{J}_2$.
Setting
$$ \cos{\theta_0} := \frac{1+3m/2-2}{1+3m/2-\sigma+\epsilon} $$
for $\theta_0\in(0,\pi/2)$, we have $\mathcal{J}_2 = (\pi-\theta_0, \pi+\theta_0)$ and $\theta_0 = O(m^{-1/2})$.
Applying Lemma \ref{lem4'}, we have
\begin{align*}
&|X_1( 1+3m/2 + (1+3m/2-\sigma+\epsilon)e^{i\theta}, \chi )| \\
&\quad\quad\leq \frac{m^2}{\log{m}} \exp{\left( C'_0 \left( \frac{(\log{(q\tau)})^{-3m -2(1+3m/2-\sigma+\epsilon)\cos{\theta}}}{\log{\log{(q\tau)}}} + (\log{(q\tau)})^{1/10} \right) \right)}
\end{align*}
for some absolute constant $C'_0>0$.
Thus,
\begin{align*}
\frac{1}{2\pi} &\int_{\theta\in\mathcal{J}_2} \log{ |X_1( 1+3m/2 + (1+3m/2-\sigma+\epsilon)e^{i\theta}, \chi )| } d\theta \\
&\quad\quad\leq \theta_0\left( \log{ \frac{m^2}{\log{m}}} + C'_0 (\log{(q\tau)})^{1/10} \right) \\
&\quad\quad\quad\quad+ \frac{C'_0(\log{(q\tau)})^{-3m}}{2\pi\log{\log{(q\tau)}}} \int_{\pi-\theta_0}^{\pi+\theta_0} (\log{(q\tau)})^{-2(1+3m/2-\sigma+\epsilon)\cos{\theta}} d\theta\\
&\quad\quad\leq \theta_0\left( \log{ \frac{m^2}{\log{m}}} + C'_0 (\log{(q\tau)})^{1/10} \right) \\
&\quad\quad\quad\quad+ \frac{C'_0(\log{(q\tau)})^{-3m}}{2\pi\log{\log{(q\tau)}}} \int_0^{2\pi} (\log{(q\tau)})^{-2(1+3m/2-\sigma+\epsilon)\cos{\theta}} d\theta\\
&\quad\quad= \theta_0\left( \log{ \frac{m^2}{\log{m}}} + C'_0 (\log{(q\tau)})^{1/10} \right) \\
&\quad\quad\quad\quad+ \frac{C'_0(\log{(q\tau)})^{-3m}}{\log{\log{(q\tau)}}} I_0( 2 (1+3m/2-\sigma+\epsilon) \log{\log{(q\tau)}} ),
\end{align*}
where $I_\nu$ is the Bessel function.
Since
$$ I_0(x) = \frac{e^x}{\sqrt{2\pi x}} (1+o(1)), $$
there exists a constant $C'_1 > 0$ such that
$$
I_0( 2 (1+3m/2-\sigma+\epsilon) \log{\log{(q\tau)}} ) \leq C'_1\frac{(\log{(q\tau)})^{2(1+3m/2-\sigma+\epsilon)}}{(m\log{\log{(q\tau)}})^{1/2}}.
$$
Hence,
\begin{align*}
\frac{1}{2\pi} &\int_{\theta\in\mathcal{J}_2} \log{ |X_1( 1+3m/2 + (1+3m/2-\sigma+\epsilon)e^{i\theta}, \chi )| } d\theta \\
&\quad\quad\ll \frac{1}{m^{1/2}} \left( (\log{(q\tau)})^{1/10} + \frac{(\log{(q\tau)})^{2(1-\sigma+\epsilon)}}{(\log{\log{(q\tau)}})^{3/2}} \right).
\end{align*}

Concluding the above, we have
\begin{align*}
\arg{G_1(\sigma \pm i\tau, \chi)} &\ll n_{X_1}(1+3m/2-\sigma,\chi) \\
&\ll \frac{m}{\epsilon} \left( 1+ \frac{1}{m^{1/2}} \left( (\log{(q\tau)})^{1/10} + \frac{(\log{(q\tau)})^{2(1-\sigma+\epsilon)}}{(\log{\log{(q\tau)}})^{3/2}} \right) \right).
\end{align*}
Taking $\epsilon = (\log{\log{(q\tau)}})^{-1}$ completes the proof.
\end{proof}


\section{Proof of theorems}
\label{sec:l_prfs}

\subsection{Evaluation of the main terms}
\quad\\ \vskip-3mm

We first prove two propositions which state out the main terms of the equations in our main theorems.
We use the functions $F(s,\chi)$ and $G_1(s,\chi)$ defined in the previous section (see equations \eqref{eq:F} and \eqref{eq:G_1}).\\


The following proposition states out the main term of the equation in Theorem \ref{cha1'}.

\begin{proposition} \label{prop2'}
Assume GRH. Take $t_q$ as in \eqref{eq:t_q}, and set $a_q := 4m$.
From Proposition \ref{zfr-right}, we note that $L'(s,\chi) \neq 0$ when $\sigma \geq a_q$.
Then for $T \geq t_q$ which satisfies $L(\sigma\pm iT,\chi)\neq0$ and $L'(\sigma\pm iT,\chi)\neq0$ for any $\sigma\in\mathbb{R}$, we have
\begin{align*}
\sum_{\substack{\rho' = \beta' + i\gamma',\\ t_q < \pm \gamma' \leq T}} \left( \beta' - \frac{1}{2} \right)
&= \frac{T}{2\pi} \log{\log{\frac{qT}{2\pi}}} + \frac{T}{2\pi} \left( \frac{1}{2} \log{m} - \log{\log{m}} \right)
- \frac{1}{q} \operatorname{Li} \left( \frac{qT}{2\pi} \right) \\
&\quad\quad \mp \frac{1}{2\pi} \int_{1/2}^{a_q} \left( -\arg{L(\sigma \pm it_q,\chi)} + \arg{G_1(\sigma \pm it_q,\chi)} \right) d\sigma \\
&\quad\quad \pm \frac{1}{2\pi} \int_{1/2}^{a_q} \left( -\arg{L(\sigma \pm iT,\chi)} + \arg{G_1(\sigma \pm iT,\chi)} \right) d\sigma \\
&\quad\quad+ O(\log{\log{q}}) +O(m),
\end{align*}
where the logarithmic branches are taken as in section \ref{sec:lems3-5}.
\end{proposition}

\begin{proof}
We first set $a_q := 4m$ and take $t_q$ as in \eqref{eq:t_q}.
We again note that $t_q \ll t_1 \ll 1$.
We also take $\sigma_1$ which satisfies Lemma \ref{lem1.2} and fix it.
Take $T \geq t_q$ such that $L(\sigma\pm iT,\chi)\neq0$ and $L'(\sigma\pm iT,\chi)\neq0$ for all $\sigma\in\mathbb{R}$.
Let $\delta \in (0, 1/2]$ and put $b := 1/2 - \delta$.

Applying Littlewood's lemma (cf. \cite[Section 3.8]{tit1}) to $G_1(s,\chi)$ on the rectangles with vertices $b \pm it_q$, $a_q \pm it_q$, $a_q \pm iT$, and $b \pm iT$, we obtain
\begin{equation} \label{eq:littlewood_1}
\begin{aligned}
&2\pi \sum_{\substack{\rho' = \beta' + i\gamma',\\ t_q < \pm \gamma' \leq T}} (\beta' - b) \\
&\quad\quad= \int_{t_q}^T \log{|G_1(b \pm it, \chi)|} dt - \int_{t_q}^T \log{|G_1(a_q \pm it, \chi)|} dt \\
&\quad\quad\quad\quad \mp \int_b^{a_q} \arg{G_1(\sigma \pm it_q, \chi)} d\sigma \pm \int_b^{a_q} \arg{G_1(\sigma \pm iT, \chi)} d\sigma \\
&\quad\quad=: I^\pm_1 + I^\pm_2 + I^\pm_3 + I^\pm_4.
\end{aligned}
\end{equation}
Applying the first inequality in Lemma \ref{lem1.1}, we can show that $I^+_2 = I^-_2 = O(m)$.
Below we estimate $I^+_1$.
\begin{equation} \label{eq:i_1}
\begin{aligned}
I^+_1 &= \int_{t_q}^T \log{|G_1(b+it,\chi)|} dt = \int_{t_q}^T \log{\left( \frac{m^b}{\log{m}}|L'(b+it,\chi)| \right)} dt \\
&= \int_{t_q}^T \log{\frac{m^b}{\log{m}}} dt + \int_{t_q}^T \log{|L'(b+it,\chi)|} dt \\
&= (b\log{m} - \log{\log{m}})T + \int_{t_q}^T \log{|F(b+it,\chi)|} dt \\
&\quad\quad+ \int_{t_q}^T \log{\left| \frac{F'}{F}(b+it,\chi) \right|} dt + \int_{t_q}^T \log{|L(1-b-it,\overline{\chi})|} dt \\
&\quad\quad+ \int_{t_q}^T \log{\left| 1 - \frac{1}{\frac{F'}{F}(b+it,\chi)} \frac{L'}{L}(1-b-it,\overline{\chi}) \right|} dt
+ O(t_q \log{m}) \\
&=: I_{11} + I_{12} + I_{13} + I_{14} + I_{15} + O(\log{m}).
\end{aligned}
\end{equation}
Here we recall that $t_q = O(1)$ from our choice of $t_q$ in \eqref{eq:t_q}.

From \eqref{eq:logF} and Stirling's formula \eqref{eq:stirling}, we have
\begin{align*}
I_{12} &= \int_{t_q}^T \log{|F(b+it,\chi)|} dt
= \int_{t_q}^T \left( \left( \frac{1}{2} - b \right) \log{\frac{qt}{2\pi}} + O\left( \frac{1}{t^2} \right) \right) dt \\
&= \left( \frac{1}{2} - b \right) \left( T\log{\frac{qT}{2\pi}} - T - t_q\log{\frac{qt_q}{2\pi}} + t_q \right) + O(1).
\end{align*}
Lemma \ref{logdevF} gives us
$$
\log{\left| \frac{F'}{F}(b+it,\chi) \right|} = \operatorname{Re} \left( \log{ \frac{F'}{F}(b+it,\chi) } \right)
= \log{\log{\frac{q|t|}{2\pi}}} + O\left( \frac{1}{t^2\log{(q|t|)}} \right),
$$
thus we have
\begin{align*}
I_{13} &= \int_{t_q}^T \log{\left| \frac{F'}{F}(b+it,\chi) \right|} dt
= \int_{t_q}^T \left( \log{\log{\frac{qt}{2\pi}}} + O\left( \frac{1}{t^2\log{(q|t|)}} \right) \right) dt \\
&= T\log{\log{\frac{qT}{2\pi}}} - t_q\log{\log{\frac{qt_q}{2\pi}}} - \int_{t_q}^T \frac{1}{\log{\frac{qt}{2\pi}}} dt + O(1) \\
&= T\log{\log{\frac{qT}{2\pi}}} - t_q\log{\log{\frac{qt_q}{2\pi}}} - \frac{2\pi}{q} \operatorname{Li} \left( \frac{qT}{2\pi} \right) + O\left( t_q \right) \\
&= T\log{\log{\frac{qT}{2\pi}}} - \frac{2\pi}{q} \operatorname{Li} \left( \frac{qT}{2\pi} \right) + O\left( \log{\log{q}} \right).
\end{align*}

Next, we estimate $I_{14}$.
We note that $\overline{L(\overline{s},\overline{\chi})} = L(s,\chi)$, hence $|L(1-b-it,\overline{\chi})| = |L(1-b+it,\chi)|$.
Take the logarithmic branch of $\log{L(s,\chi)}$ so that $\log{L(s,\chi)} = \sum_{n=2}^\infty \chi(n)\Lambda(n) (\log{n})^{-1} n^{-s}$ holds for Re$(s)>1$ and that it is holomorphic in $\mathbb{C} \backslash \{ \rho+\lambda \mid L(\rho,\chi) = 0, \lambda\leq0 \}$.
Then applying Cauchy's integral theorem to $\log{L(s,\chi)}$ on the rectangle with vertices $1-b + it_q$, $a_q + it_q$, $a_q + iT$, $1-b + iT$ and taking the imaginary part, we can show that
\begin{align*}
I_{14} &= \int_{t_q}^T \log{|L(1-b-it,\overline{\chi})|} dt = \int_{t_q}^T \log{|L(1-b+it,\chi)|} dt \\
&= \int_{1-b}^{a_q} \arg{L(\sigma+it_q,\chi)} d\sigma - \int_{1-b}^{a_q} \arg{L(\sigma+iT,\chi)} d\sigma +O(1).
\end{align*}

Finally we estimate $I_{15}$.
Since $L(s,\chi) = F(s,\chi)L(1-s,\overline{\chi})$, we have
\begin{equation} \label{eq:i_15_eq}
\frac{1}{\frac{F'}{F}(s,\chi)} \frac{L'}{L}(s,\chi) = 1 - \frac{1}{\frac{F'}{F}(s,\chi)} \frac{L'}{L}(1-s,\overline{\chi}).
\end{equation}
From Lemma \ref{lem1.3}, the function on the left-hand side of \eqref{eq:i_15_eq} is holomorphic and has no zeros in $\sigma_1 \leq \sigma < 1/2, |t| \geq t_1 - 1$.
From Lemma \ref{lem1.2}, the function on the right-hand side of \eqref{eq:i_15_eq} is holomorphic and has no zeros in $\sigma \leq \sigma_1, |t| \geq 2$.
Thus we can determine
$$
\log{\left( 1 - \frac{1}{\frac{F'}{F}(s,\chi)} \frac{L'}{L}(1-s,\overline{\chi}) \right)}
$$
so that it tends to $0$ as $\sigma \rightarrow -\infty$ which follows from Lemma \ref{lem1.2}, and that it is holomorphic in $\sigma < 1/2, |t| > t_q - 1 (> t_1 - 1)$.
Now we apply Cauchy's integral theorem to it on the trapezoid with vertices $-t_q + it_q$, $b + it_q$, $b + iT$, and $-T + iT$.
Lemma \ref{lem1.2} allows us to show
$$
\left( \int_{\sigma_1+iT}^{-T+iT} + \int_{-T+iT}^{-t_q+it_q} + \int_{-t_q+it_q}^{\sigma_1+it_q} \right)
\log{\left( 1 - \frac{1}{\frac{F'}{F}(s,\chi)} \frac{L'}{L}(1-s,\overline{\chi}) \right)} ds = O(1).
$$
Thus taking the imaginary part, we obtain
\begin{align*}
&\int_{t_q}^T \log{\left| 1 - \frac{1}{\frac{F'}{F}(b+it,\chi)} \frac{L'}{L}(1-b-it,\overline{\chi}) \right|} dt \\
&\quad\quad
= \int_{\sigma_1}^b \arg{\left( 1 - \frac{1}{\frac{F'}{F}(\sigma+iT,\chi)} \frac{L'}{L}(1-\sigma-iT,\overline{\chi}) \right)} d\sigma \\
&\quad\quad\quad\quad
- \int_{\sigma_1}^b \arg{\left( 1 - \frac{1}{\frac{F'}{F}(\sigma+it_q,\chi)} \frac{L'}{L}(1-\sigma-it_q,\overline{\chi}) \right)} d\sigma + O(1)
\end{align*}
Now we let
$$
\log{\left( \frac{1}{\frac{F'}{F}(s,\chi)} \frac{L'}{L}(s,\chi) \right)} = \log{\left( 1 - \frac{1}{\frac{F'}{F}(s,\chi)} \frac{L'}{L}(1-s,\overline{\chi}) \right)}
$$
and determine the logarithmic branch of $\log{(F'/F)(s,\chi)}$ and $\log{(L'/L)(s,\chi)}$ in the region $\sigma_1 \leq \sigma < 1/2, |t| \geq t_q - 1$ as in Lemma \ref{lem1.3}.
Note that both of them and the functions on both sides of \eqref{eq:i_15_eq} are all continuous with respect to $s$ in $\sigma_1 \leq \sigma < 1/2, |t| \geq t_q - 1$. Furthermore, the two regions $\sigma_1 \leq \sigma < 1/2, t \geq t_q - 1$ and $\sigma_1 \leq \sigma < 1/2, -t \geq t_q - 1$ are connected. Thus we have
$$
\arg{\left( 1 - \frac{1}{\frac{F'}{F}(s,\chi)} \frac{L'}{L}(1-s,\overline{\chi}) \right)} = -\arg{\frac{F'}{F}(s,\chi)} + \arg{\frac{L'}{L}(s,\chi)} + 2\pi n_q
$$
for some $n_q\in\mathbb{Z}$ that depends only at most on $q$.
From our choice of logarithmic branch, we have $n_q = 0$. Thus,
\begin{equation} \label{eq:ineq_15}
-\frac{2\pi}{3} < \arg{\left( 1 - \frac{1}{\frac{F'}{F}(s,\chi)} \frac{L'}{L}(1-s,\overline{\chi}) \right)} < \frac{2\pi}{3}
\end{equation}
for $\sigma_1 \leq \sigma < 1/2, |t| \geq t_q - 1$.
Therefore we obtain
$$
I_{15} = \int_{t_q}^T \log{\left| 1 - \frac{1}{\frac{F'}{F}(b+it,\chi)} \frac{L'}{L}(1-b-it,\overline{\chi}) \right|} dt = O(1).
$$

Collecting the above calculations, we have
\begin{align*}
I^+_1 &= T\log{\log{\frac{qT}{2\pi}}} + \left( b\log{m} - \log{\log{m}} \right)T
- \frac{2\pi}{q} \operatorname{Li} \left( \frac{qT}{2\pi} \right) \\
&\quad\quad+ \left( \frac{1}{2} - b \right) \left( T\log{\frac{qT}{2\pi}} - T - t_q\log{\frac{qt_q}{2\pi}} + t_q \right) \\
&\quad\quad + \int_{1-b}^{a_q} \arg{L(\sigma+it_q,\chi)} d\sigma
- \int_{1-b}^{a_q} \arg{L(\sigma+iT,\chi)} d\sigma + O( \log{\log{q}} ).
\end{align*}
Similarly, we can show that
\begin{align*}
I^-_1 &= T\log{\log{\frac{qT}{2\pi}}} + \left( b\log{m} - \log{\log{m}} \right)T
- \frac{2\pi}{q} \operatorname{Li} \left( \frac{qT}{2\pi} \right) \\
&\quad\quad+ \left( \frac{1}{2} - b \right) \left( T\log{\frac{qT}{2\pi}} - T - t_q\log{\frac{qt_q}{2\pi}} + t_q \right) \\
&\quad\quad- \int_{1-b}^{a_q} \arg{L(\sigma-it_q,\chi)} d\sigma
+ \int_{1-b}^{a_q} \arg{L(\sigma-iT,\chi)} d\sigma + O( \log{\log{q}} ).
\end{align*}

Thus we have
\begin{align*}
2\pi \sum_{\substack{\rho' = \beta' + i\gamma',\\ t_q < \pm \gamma' \leq T}} (\beta' - b)
&= T\log{\log{\frac{qT}{2\pi}}} + \left( b\log{m} - \log{\log{m}} \right)T
- \frac{2\pi}{q} \operatorname{Li} \left( \frac{qT}{2\pi} \right) \\
&\quad\quad+ \left( \frac{1}{2} - b \right) \left( T\log{\frac{qT}{2\pi}} - T - t_q\log{\frac{qt_q}{2\pi}} + t_q \right) \\
&\quad\quad \pm \int_{1-b}^{a_q} \arg{L(\sigma \pm it_q, \chi)} d\sigma
\mp \int_{1-b}^{a_q} \arg{L(\sigma \pm iT, \chi)} d\sigma \\
&\quad\quad \mp \int_b^{a_q} \arg{G_1(\sigma \pm it_q, \chi)} d\sigma
\pm \int_b^{a_q} \arg{G_1(\sigma\pm iT, \chi)} d\sigma \\
&\quad\quad+ O( \log{\log{q}} ) + O(m).
\end{align*}
Taking $\delta \rightarrow 0$, we obtain Proposition \ref{prop2'}.
\end{proof}


The following proposition states out the main term of the equation in Theorem \ref{cha2'}.

\begin{proposition} \label{prop6'}
Assume GRH. Take $t_q$ as in \eqref{eq:t_q}. Then for $T \geq 2$ which satisfies $L(\sigma \pm iT, \chi) \neq 0$ and $L'(\sigma \pm iT, \chi) \neq 0$ for all $\sigma\in\mathbb{R}$, we have
\begin{align*}
N_1(T,\chi) &= \frac{T}{\pi} \log{\frac{qT}{2m\pi}} - \frac{T}{\pi}
- A(t_q,\chi) - B(t_q,\chi) + A(T,\chi) + B(T,\chi) \\
&\quad\quad+ A(-t_q,\chi) + B(-t_q,\chi) - A(-T,\chi) - B(-T,\chi) + O(m^{1/2}\log{q}),
\end{align*}
where
$$ A(\tau,\chi) := \frac{1}{2\pi} \arg{G_1\left(\frac{1}{2}+i\tau,\chi\right)}, \quad B(\tau,\chi) := \frac{1}{2\pi} \arg{L\left(\frac{1}{2}+i\tau,\chi\right)}, $$
and the logarithmic branches are taken as in section \ref{sec:lems3-5}.
\end{proposition}

\begin{proof}
Take $a_q, \sigma_1, t_q, T, \delta, b$ as in the beginning of the proof of Proposition \ref{prop2'}.
Let $b' := 1/2 - \delta/2$.
Replacing $b$ by $b'$ in \eqref{eq:littlewood_1}, we have
\begin{align*}
2\pi \sum_{\substack{\rho' = \beta' + i\gamma',\\ t_q < \pm \gamma' \leq T}} (\beta' - b')
&= \int_{t_q}^T \log{| G_1(b' \pm it, \chi) |} dt - \int_{t_q}^T \log{| G_1(a_q \pm it, \chi) |} dt \\
&\quad\quad \mp \int_{b'}^{a_q} \arg{G_1(\sigma \pm it_q, \chi)} d\sigma
\pm \int_{b'}^{a_q} \arg{G_1(\sigma \pm iT, \chi)} d\sigma.
\end{align*}
Subtracting these from \eqref{eq:littlewood_1}, we obtain
\begin{align*}
\delta\pi \sum_{\substack{\rho' = \beta' + i\gamma',\\ t_q < \pm \gamma' \leq T}} 1
&= \int_{t_q}^T \left( \log{| G_1(b \pm it, \chi) |} - \log{| G_1(b' \pm it, \chi) |} \right) dt \\
&\quad\quad \mp \int_b^{b'} \arg{G_1(\sigma \pm it_q, \chi)} d\sigma
\pm \int_b^{b'} \arg{G_1(\sigma \pm iT, \chi)} d\sigma \\
&=: J^\pm_1 + J^\pm_2 + J^\pm_3.
\end{align*}

We estimate $J^\pm_1$.
From \eqref{eq:i_1}, we have
\begin{align*}
J^+_1 &= \int_{t_q}^T \left( \log{|G_1(b+it,\chi)|} - \log{|G_1(b'+it,\chi)|} \right) dt \\
&= (b-b')(T-t_q)\log{m} + \int_{t_q}^T ( \log{|F(b+it,\chi)|} - \log{|F(b'+it,\chi)|} )dt \\
&\quad\quad
+ \int_{t_q}^T \left( \log{\left| \frac{F'}{F}(b+it,\chi) \right|} - \log{\left| \frac{F'}{F}(b'+it,\chi) \right|} \right) dt \\
&\quad\quad
+ \int_{t_q}^T \left( \log{|L(1-b-it,\overline{\chi})|} - \log{|L(1-b'-it,\overline{\chi})|} \right) dt \\
&\quad\quad
+ \int_{t_q}^T \Bigg( \log{\left| 1 - \frac{1}{\frac{F'}{F}(b+it,\chi)} \frac{L'}{L}(1-b-it,\overline{\chi}) \right|} \\
&\quad\quad\quad\quad
- \log{\left| 1 - \frac{1}{\frac{F'}{F}(b'+it,\chi)} \frac{L'}{L}(1-b'-it,\overline{\chi}) \right|} \Bigg) dt \\
&=: J_{11} + J_{12} + J_{13} + J_{14} + J_{15}.
\end{align*}
Applying Cauchy's theorem to $\log{F(s,\chi)}$ on the rectangle $C$ with vertices $b+it_q$, $b'+it_q$, $b'+iT$, $b+iT$, and taking the imaginary part, we have
$$
J_{12} = \int_b^{b'} \arg{F(\sigma+it_q,\chi)} d\sigma -  \int_b^{b'} \arg{F(\sigma+iT,\chi)} d\sigma.
$$
From \eqref{eq:logF}, we can show that
$$
J_{12} =  \left( T\log{\frac{qT}{2\pi}} - T \right) \frac{\delta}{2} - \left( t_q\log{\frac{qt_q}{2\pi}} - t_q \right) \frac{\delta}{2} + O(\delta)
$$
Next, we take the logarithmic branch of $\log{(F'/F)(s,\chi)}$ as in condition (1) of Lemma \ref{lem1.3}.
Applying Cauchy's integral theorem to $\log{(F'/F)(s,\chi)}$ on $C$ taking the imaginary part, we have
$$
J_{13} = \int_b^{b'} \arg{\frac{F'}{F}(\sigma+it_q,\chi)} d\sigma -  \int_b^{b'} \arg{\frac{F'}{F}(\sigma+iT,\chi)} d\sigma = O(\delta)
$$
To estimate $J_{14}$, we define a branch of $\log{L(s,\chi)}$ as in the estimation of $I_{14}$ in the proof of Proposition \ref{prop2'} and apply Cauchy's integral theorem on the rectangle with vertices $1-b'+it_q$, $1-b+it_q$, $1-b+iT$, $1-b'+iT$.
Taking the imaginary part we obtain
$$
J_{14} = -\int_{1-b'}^{1-b} \arg{L(\sigma+it_q,\chi)} d\sigma +  \int_{1-b'}^{1-b} \arg{L(\sigma+iT,\chi)} d\sigma.
$$
Finally, we define a branch of
$$
\log{\left( 1 - \frac{1}{\frac{F'}{F}(s,\chi)} \frac{L'}{L}(1-s,\overline{\chi}) \right)}
$$
as in the estimation of $I_{15}$ in the proof of Proposition \ref{prop2'} and apply Cauchy's integral theorem to it on $C$.
Taking the imaginary part, we have
\begin{align*}
J_{15} &= \int_b^{b'} \arg{\left( 1 - \frac{1}{\frac{F'}{F}(\sigma+it_q,\chi)} \frac{L'}{L}(1-\sigma-it_q,\overline{\chi}) \right)} d\sigma \\
&\quad\quad
- \int_b^{b'} \arg{\left( 1 - \frac{1}{\frac{F'}{F}(\sigma+iT,\chi)} \frac{L'}{L}(1-\sigma-iT,\overline{\chi}) \right)} d\sigma \\
&= O(\delta)
\end{align*}
by inequalities \eqref{eq:ineq_15}.
Then we estimate $J^-_1$ similarly.

We then obtain
\begin{align*}
\delta\pi \sum_{\substack{\rho' = \beta' + i\gamma',\\ t_q < \pm \gamma' \leq T}} 1
&= -(T-t_q) \frac{\delta}{2} \log{m} + \left( T\log{\frac{qT}{2\pi}} - T \right) \frac{\delta}{2} - \left( t_q\log{\frac{qt_q}{2\pi}} - t_q \right) \frac{\delta}{2} \\
&\quad\quad \mp \int_{1-b'}^{1-b} \arg{L(\sigma \pm it_q, \chi)} d\sigma
\pm \int_{1-b'}^{1-b} \arg{L(\sigma \pm iT, \chi)} d\sigma \\
&\quad\quad \mp \int_b^{b'} \arg{G_1(\sigma \pm it_q, \chi)} d\sigma
\pm \int_b^{b'} \arg{G_1(\sigma \pm iT, \chi)} d\sigma
+ O(\delta).
\end{align*}

Taking the limit $\delta\rightarrow0$ and applying the mean value theorem, for $\tau = \pm t_q$ and $\tau = \pm T$ we have
$$
\lim_{\delta\rightarrow0} \frac{1}{\pi\delta} \int_{1-b'}^{1-b} \arg{L(\sigma + i\tau, \chi)} d\sigma
= B(\tau, \chi)
$$
and
$$
\lim_{\delta\rightarrow0} \frac{1}{\pi\delta} \int_{1-b'}^{1-b} \arg{G_1(\sigma + i\tau, \chi)} d\sigma
= A(\tau, \chi)
$$
by noting that $b = 1/2 - \delta$ and $b' = 1/2 - \delta/2$.
Hence,
\begin{align*}
N_1(T,\chi) - N_1(t_q,\chi) &= \frac{T}{\pi} \log{\frac{qT}{2m\pi}} - \frac{T}{\pi}
- \left( \frac{t_q}{\pi} \log{\frac{qt_q}{2m\pi}} - \frac{t_q}{\pi} \right) \\
&\quad\quad- A(t_q,\chi) - B(t_q,\chi) + A(T,\chi) + B(T,\chi) \\
&\quad\quad+ A(-t_q,\chi) + B(-t_q,\chi) - A(-T,\chi) - B(-T,\chi) + O(1).
\end{align*}
Referring to \cite[Theorem 5]{as}), we see that
\begin{equation} \label{eq:n_1tau}
N_1(t_q,\chi) = \frac{t_q}{\pi}\log{\frac{qt_q}{2m\pi}} - \frac{t_q}{\pi} + O\left( m^{1/2}\log{(qt_q)} \right).
\end{equation}
Hence,
\begin{align*}
N_1(T,\chi) &= \frac{T}{\pi} \log{\frac{qT}{2m\pi}} - \frac{T}{\pi}
- A(t_q,\chi) - B(t_q,\chi) + A(T,\chi) + B(T,\chi) \\
&\quad\quad+ A(-t_q,\chi) + B(-t_q,\chi) - A(-T,\chi) - B(-T,\chi) + O(m^{1/2}\log{q}).
\end{align*}

If $2 \leq T < t_q$, then $N_1(T,\chi) \leq N_1(t_q,\chi) = O(m^{1/2}\log{q})$,
which can be included in the error term.
Thus the proof is complete.
\end{proof}


\subsection{Completion of the proofs}
\quad\\ \vskip-3mm

We begin with the proof of Theorem \ref{cha1'}.
Referring to \cite[Theorem 6]{as}, we have
\begin{equation} \label{eq:proof1.1}
\sum_{\substack{\rho' = \beta' + i\gamma',\\ |\gamma'| \leq t_q}} (\beta' - 1/2)
\ll m^{1/2}\log{q}.
\end{equation}
This also implies that when $2\leq T<t_q$, 
$$
\sum_{\substack{\rho' = \beta' + i\gamma',\\ |\gamma'| \leq T}} (\beta' - 1/2)
\ll m^{1/2}\log{q}.
$$

Next, we estimate
$$
\sum_{\substack{\rho' = \beta' + i\gamma',\\ t_q<|\gamma'|\leq T}} (\beta' - 1/2).
$$
We divide the proof in two cases.\\

\underline{\textbf{Case 1:}} For $T \geq t_q$ which satisfies $L(\sigma \pm iT, \chi) \neq 0, L'(\sigma \pm iT, \chi) \neq 0$ for all $\sigma\in\mathbb{R}$. \vskip1mm

In this case, we apply Proposition \ref{prop2'} and provoke Lemmas \ref{mon13.2.11}, \ref{lem3'}, and \ref{lem5'} to obtain the error term.

We apply Lemmas \ref{mon13.2.11}, \ref{lem5'}, and \ref{lem3'} to obtain
$$
\int_{1/2}^{1/2+(\log{(q\tau)})^{-1}} \arg{L(\sigma\pm i\tau,\chi)} d\sigma \ll 1,
$$
$$
\int_3^{a_q} \arg{\frac{G_1}{L}(\sigma\pm i\tau,\chi)} d\sigma \ll m\log{m}
$$
for $\tau \geq t_q$,
and
$$
\int_{1/2}^{1/2+(\log{(qt_q)})^{-1}} \arg{G_1(\sigma\pm it_q,\chi)} d\sigma \ll \frac{m^{1/2}}{(\log{\log{q}})^{1/2}},
$$
$$
\int_{1/2}^{1/2+(\log{(qT)})^{-1}} \arg{G_1(\sigma\pm iT,\chi)} d\sigma \ll \frac{m^{1/2}}{(\log{\log{(qT)}})^{1/2}},
$$
$$
\int_{1/2+(\log{(qt_q)})^{-1}}^3 \arg{\frac{G_1}{L}(\sigma\pm it_q,\chi)} d\sigma \ll m^{1/2}(\log{\log{q}})^2 + m\log{\log{q}},
$$
$$
\int_{1/2+(\log{(qT)})^{-1}}^3 \arg{\frac{G_1}{L}(\sigma\pm iT,\chi)} d\sigma \ll m^{1/2}(\log{\log{(qT)}})^2 + m\log{\log{(qT)}}.
$$

Inserting the above estimates into the formula given in Proposition \ref{prop2'} and adding this to \eqref{eq:proof1.1}, we obtain Theorem \ref{cha1'} for Case 1. \\

\underline{\textbf{Case 2:}} For $T \geq t_q$ such that any of
$L(\sigma+iT,\chi)\neq0$, $L(\sigma-iT,\chi)\neq0$, $L'(\sigma+iT,\chi)\neq0$, or $L'(\sigma-iT,\chi)\neq0$
is \textit{not} satisfied for some $\sigma\in\mathbb{R}$. \vskip1mm

In this case, we have some increment in our sum as much as
$$
\sum_{\substack{\rho' = \beta' + i\gamma',\\ |\gamma'| = T}} (\beta' - 1/2).
$$
Now we estimate this and show that this increment can be included in the error term.
First we look for some small $0 < \epsilon < (\log{\log{(qT)}})^{-1}$ such that
$L(\sigma \pm i(T\pm\epsilon), \chi)\neq0, L'(\sigma \pm i(T\pm\epsilon), \chi)\neq0$
holds for all $\sigma\in\mathbb{R}$ and apply the method of Case 1.
Then we find that the increment
\begin{align*}
\sum_{\substack{\rho' = \beta' + i\gamma',\\ |\gamma'| = T}} (\beta' - 1/2)
&\ll \left|\sum_{\substack{\rho' = \beta' + i\gamma',\\ T-\epsilon < |\gamma'| \leq T+\epsilon}} (\beta' - 1/2)\right| \\
&= \left|\sum_{\substack{\rho' = \beta' + i\gamma',\\ |\gamma'| \leq T+\epsilon}} (\beta' - 1/2) -\sum_{\substack{\rho' = \beta' + i\gamma',\\ |\gamma'| \leq T-\epsilon}} (\beta' - 1/2)\right| \\
&\ll m^{1/2}(\log{\log{(qT)}})^2 + m\log{\log{(qT)}} + m^{1/2}\log{q}
\end{align*}
can be included in the error term.
So we also have Theorem \ref{cha1'} for this case.
\qed \\


To complete the proof of Theorem \ref{cha2'}, as in the proof of Theorem \ref{cha1'}, we also consider two cases.
In the first case, for $T \geq 2$ which satisfies $L(\sigma \pm iT, \chi) \neq 0, L'(\sigma \pm iT, \chi) \neq 0$ for all $\sigma\in\mathbb{R}$, the error terms are estimated as follows:
From Lemma \ref{lem5'}, we have
$$
\arg{G_1\left( \frac{1}{2} \pm it_q, \chi \right)} = O\left( \frac{m^{1/2} \log{q}}{(\log{\log{q}})^{1/2}} \right)
$$
and
$$
\arg{G_1\left( \frac{1}{2} \pm iT, \chi \right)} = O\left( \frac{m^{1/2} \log{(qT)}}{(\log{\log{(qT)}})^{1/2}} \right).
$$
From Lemma \ref{mon13.2.11}, we have
$$
\arg{L\left( \frac{1}{2} \pm it_q, \chi \right)} = O\left( \frac{\log{q}}{\log{\log{q}}} \right), \quad \arg{L\left( \frac{1}{2} \pm iT, \chi \right)} = O\left( \frac{\log{(qT)}}{\log{\log{(qT)}}} \right).
$$

Therefore,
$$
N_1(T,\chi) = \frac{T}{\pi} \log{\frac{qT}{2m\pi}} - \frac{T}{\pi} + O\left( \frac{m^{1/2} \log{(qT)}}{(\log{\log{(qT)}})^{1/2}} \right) + O(m^{1/2}\log{q})
$$
for this case.

In the second case, we consider for $T \geq 2$ such that any of
$L(\sigma+iT,\chi)\neq0$, $L(\sigma-iT,\chi)\neq0$, $L'(\sigma+iT,\chi)\neq0$, or $L'(\sigma-iT,\chi)\neq0$
is \textit{not} satisfied for some $\sigma\in\mathbb{R}$.
Similar to the proof of Theorem \ref{cha1'}, we look for some small $0 < \epsilon < (\log{(qT)})^{-1}$ such that
$L(\sigma \pm i(T\pm\epsilon), \chi)\neq0, L'(\sigma \pm i(T\pm\epsilon), \chi)\neq0$
holds for all $\sigma\in\mathbb{R}$.
Applying the method of the first case we obtain
\begin{equation} \label{eq:n1t_eps}
\begin{aligned}
N_1(T \pm \epsilon,\chi)
&= \frac{T \pm \epsilon}{\pi} \log{\frac{q(T \pm \epsilon)}{2m\pi}} - \frac{T \pm \epsilon}{\pi}
+ O\left( \frac{m^{1/2} \log{(qT)}}{(\log{\log{(qT)}})^{1/2}} \right) \\
&\quad\quad+ O(m^{1/2}\log{q}).
\end{aligned}
\end{equation}
Noting the inequalities
$$
N_1(T-\epsilon,\chi) \leq N_1(T,\chi) \leq N_1(T-\epsilon,\chi) + \left( N_1(T+\epsilon,\chi) - N_1(T-\epsilon,\chi) \right),
$$
from \eqref{eq:n1t_eps} we can easily deduce
$$
N_1(T,\chi) = \frac{T}{\pi} \log{\frac{qT}{2m\pi}} - \frac{T}{\pi} + O\left( \frac{m^{1/2} \log{(qT)}}{(\log{\log{(qT)}})^{1/2}} \right) + O(m^{1/2}\log{q})
$$
for this case.
\qed


\section*{Acknowledgements}

The author would like to express her gratitude to Prof. Kohji Matsumoto, Prof. Hirotaka Akatsuka, and Mr. Yuta Suzuki for their valuable advice.




\bibliographystyle{amsalpha}

\begin{thebibliography}{ABC0123}
\bibitem[Aka12]{aka} H. Akatsuka, \emph{Conditional estimates for error terms related to the distribution of zeros of $\zeta'(s)$}, J. Number Theory {\bf 132} (2012), no. 10, 2242--2257.
\bibitem[AS-p]{as} H. Akatsuka and A. I. Suriajaya, \emph{Zeros of the first derivative of Dirichlet $L$-functions}, in preparation.
\bibitem[Ber70]{ber} B. C. Berndt, \emph{The number of zeros for $\zeta^{(k)}(s)$}, J. Lond. Math. Soc. (2) {\bf 2} (1970), 577--580.
\bibitem[LM74]{lev} N. Levinson and H. L. Montgomery, \emph{Zeros of the derivatives of the Riemann zeta-function}, Acta Math. {\bf 133} (1974), 49--65.
\bibitem[MV06]{mon} H. L. Montgomery and R. C. Vaughan, \emph{Multiplicative Number Theory I: Classical Theory}, Cambridge University Press, 2006.
\bibitem[MV06-cor]{mon.er} H. L. Montgomery and R. C. Vaughan, \emph{Errata of Multiplicative Number Theory I: Classical Theory}, available at: http://www-personal.umich.edu/~hlm/mnt1err.pdf.
\bibitem[Sel46]{sel} A. Selberg, \emph{Contributions to the theory of Dirichlet's $L$-functions}, Skr. Norske Vid-Akad. Oslo. I. {\bf 1946} (1946), no. 3, 1--62.
\bibitem[Spe35]{spe} A. Speiser, \emph{Geometrisches zur Riemannschen Zetafunktion}, Math. Ann. {\bf 110} (1935), no. 1, 514--521.
\bibitem[Sur15]{ade} A. I. Suriajaya, \emph{On the zeros of the $k$-th derivative of the Riemann zeta function under the Riemann Hypothesis}, Funct. Approx. Comment. Math. {\bf 53} (2015), no.1, 69--95.
\bibitem[Tit39]{tit1} E. C. Titchmarsh, \emph{The theory of functions}, second ed., Oxford Univ. Press, 1939.
\bibitem[Y{\i}l96]{yil} C. Y. Y{\i}ld{\i}r{\i}m, \emph{Zeros of derivatives of Dirichlet $L$-functions}, Turkish J. Math. {\bf 20} (1996), 521--534.
\end{thebibliography}

\end{document}